\documentclass[12pt]{amsart}
\usepackage{amsmath}
\usepackage{amssymb}
\usepackage{amsfonts}
\usepackage{latexsym}
\usepackage{amscd}
\usepackage[mathscr]{euscript}
\usepackage{enumitem}
\usepackage{xy} \xyoption{all}
\usepackage{pdfsync}
\usepackage[active]{srcltx}

\newcommand{\N}{{\mathbb{N}}}

\newcommand{\Z}{{\mathbb{Z}}}

\newcommand{\C}{{\mathbb{C}}}
\newcommand{\D}{{\mathcal{D}}}

\newcommand{\uloopr}[1]{\ar@'{@+{[0,0]+(-4,5)}@+{[0,0]+(0,10)}@+{[0,0] +(4,5)}}^{#1}}
\newcommand{\uloopd}[1]{\ar@'{@+{[0,0]+(5,4)}@+{[0,0]+(10,0)}@+{[0,0]+ (5,-4)}}^{#1}}
\newcommand{\dloopr}[1]{\ar@'{@+{[0,0]+(-4,-5)}@+{[0,0]+(0,-10)}@+{[0, 0]+(4,-5)}}_{#1}}
\newcommand{\dloopd}[1]{\ar@'{@+{[0,0]+(-5,4)}@+{[0,0]+(-10,0)}@+{[0,0 ]+(-5,-4)}}_{#1}}

\newcommand{\luloop}[1]{\ar@'{@+{[0,0]+(-8,2)}@+{[0,0]+(-10,10)}@+{[0, 0]+(2,2)}}^{#1}}

\newtheorem{lem}{Lemma}[section]
\newtheorem{corol}[lem]{Corollary}
\newtheorem{theor}[lem]{Theorem}
\newtheorem{prop}[lem]{Proposition}
\theoremstyle{definition}
\newtheorem{defi}[lem]{Definition}

\newtheorem{exem}[lem]{Example}

\newtheorem{rema}[lem]{Remark}

\begin{document}
\title[Purely infinite crossed products by endomorphisms]{Purely infinite crossed products by endomorphisms}%
\author{Eduard Ortega}
\address{Department of Mathematical Sciences\\
NTNU\\
NO-7491 Trondheim\\
Norway } \email{eduard.ortega.esparza@gmail.com}

\author{Enrique Pardo}
\address{Departamento de Matem\'aticas, Facultad de Ciencias\\ Universidad de C\'adiz, Campus de
Puerto Real\\ 11510 Puerto Real (C\'adiz)\\ Spain.}
\email{enrique.pardo@uca.es}\urladdr{https://sites.google.com/a/gm.uca.es/enrique-pardo-s-home-page/}\urladdr{http://www.uca.es/dpto/C101/pags-personales/enrique.pardo}

\thanks{This research was supported by the NordForsk Research Network
  ``Operator Algebras and Dynamics'' (grant \#11580). The second author was partially supported by PAI III grants FQM-298 and P07-FQM-7156 of the Junta de Andaluc\'{\i}a. Both authors were partially supported by the DGI-MICINN and European Regional Development Fund, jointly, through Project MTM2011-28992-C02-02 and by 2009 SGR 1389 grant of the Comissionat per Universitats i Recerca de la Generalitat de
Catalunya.} 
%\date{\today}
\dedicatory{To Ingvar Ortega Redalen}
\begin{abstract}
We study the crossed product $C^*$-algebra associated to injective endomorphisms, which turns out to be equivalent to study the crossed product by the dilated autormorphism. We prove that the dilation of the  Bernoulli $p$-shift endomorphism is topologically free. As a consequence, we have a way to twist any endomorphism of a $\D$-absorbing $C^*$-algebra into one whose dilated automorphism is essentially free and have the same $K$-theory map than the original one. This allows us to construct purely infinite crossed products $C^*$-algebras with diverse ideal structures.    
\end{abstract}

\maketitle

\section*{Introduction}

The study of group actions on $C^*$-algebras has been largely developed by several authors during the last years. To this end, a key strategy is to characterize properties of the crossed product $C^*$-algebra by looking at the dynamical properties of the action. This is very intuitive in the commutative case, but becomes more subtle when the $C^*$-algebra is non-commutative. An intensive work on that area was done by Olesen and Pedersen in \cite{OP3}, where they explored a certain subset of the dual group $\Gamma$ of $G$, called the \emph{Connes' Spectrum}, which measured the obstruction of the automorphism to be inner.  In this way, they were able to characterize when certain crossed products $C^*$-algebra were simple. An associated problem is to examine conditions on the group action so that the ideals in the 
crossed product are separated by the base algebra; this allows us to have control on the ideal structure of the crossed product. Renault \cite{Re} stated implicitly that \emph{essential freeness} of a group $G$ acting on a $C^*$-algebra  $A$ might be enough to guarantee that $A$ separates the ideals of the reduced crossed product $A\rtimes_rG$. Sierakowski \cite{Sie} studied this problem and presented another way of ensuring this separation property on $A\rtimes_r G$. For, he defined a generalized version of the Rokhlin property, which he called the residual Rokhlin* property. Hence, he showed that $A$ separates the ideals in $A\rtimes_rG$ provided that the action of $G$ on $A$ is exact and satisfy the residual Rokhlin* property.

In \cite{Cu}, Cuntz defined the fundamental Cuntz algebras $\mathcal{O}_n$. He also represented these algebras as crossed products of a UHF-algebra by an endomorphism, and he used this representation to prove the simplicity of these algebras.  He saw this construction as a full corner of an ordinary crossed product. Later, Paschke \cite{Pa} gave an elegant generalization of Cuntz's results, and described the crossed product of a unital $C^*$-algebra by an endomorphism $\alpha:A\rightarrow A$, written $A\times_\alpha \N$, as the $C^*$-algebra generated by $A$ and an isometry $V$, such that $VaV^*=\alpha(a)$. Endomorphisms of $C^*$-algebras appeared elsewhere, and this led Stacey to give a modern description of their crossed products in terms of covariant representations and universal properties \cite{Stacey}. As noticed by Cuntz, when the endomorphism is injective, it is possible to transform it into automorphism and the isometries into unitaries via a direct limit construction. So, the crossed product by an endomorphism can be seen as a full corner of a crossed product by an automorphism. This allows to import results from the well-developed theory of crossed products by groups. That fact can be extended to actions by cancellative semigroups; for a more general exposition look at \cite{Laca} and the references therein.

Along his work, Cuntz defined purely infinite simple $C^*$-algebras, and he proved that the Cuntz algebras $\mathcal{O}_n$ are among those $C^*$-algebras. Purely infinite simple Cuntz-Krieger algebras (a generalization of Cuntz algebras introduced in \cite{CK}) where classified by using K-Theoretical invariants \cite{Ror}. However, the most spectacular result about purely $C^*$-algebras came from works of E. Kirchberg and N.C. Phillips \cite{Kie, Phi}, who proved that separable, unital purely infinite simple $C^*$-algebras in the bootstrap  class are classified  (up to isomorphism) by their $K$-theory. There were some definitions that generalize the notion of purely infinite $C^*$-algebras to the non-simple case. The most accepted, and useful, was due to Kirchberg and R{\o}rdam \cite{KR}.  In \cite{Kie} Kirchberg gave also some nice classification results for non-simple purely infinite $C^*$-algebras through  bivariant $KK$-theory. \vspace{.2truecm}

Our main goals in this paper are to give conditions on an endomorphism $\alpha\in\text{End}(A)$ to guarantee that: (1) $A$ separates ideals of $A\times_\alpha \N$; (2) $A\times_\alpha \N$ is purely infinite. Our fundamental technique is seeing $A\times_\alpha \N$ as a full corner of a crossed product of  $\bar{A}\times_{\bar{\alpha}}\Z$, where $\bar{A}$ is the dilation of $A$ by $\alpha$. After reducing the situation of a crossed product by an automorphism, we will use results of Sierakowski \cite{Sie}, and of Pasnicu and  R{\o}rdam \cite{PR}. We will also give a concrete example of endomorphism satisfying the residual Rokhlin* property: it is the so called shift endomorphism, that was constructed by Cuntz \cite{Cu} and later studied by Dykema and R{\o}rdam \cite{DR}.

The contents of this paper can be summarized as follows: In Section $1$, we introduce basic notation, and we construct the dilation of an injective endomorphism. We then review results on actions by a single automorphism, i.e. actions of $\Z$, and we show equivalent conditions for this action being topological and essentially free. We finally recall the residual Rokhlin* property defined by Sierakowski. In Section $2$, we prove that the dilation of the shift endomorphism is topologically free. In particular, given a strongly self-absorbing  $C^*$-algebra with a non-trivial projection $\D$, the shift endomorphism is topologically free. Even more, this allows us to twist any endomorphism on a $\D$-absorbing $C^*$-algebra into one that is essentially free and that induces the same $K$-theory map. In Section $3$, we give conditions for a crossed product by a single endomorphism being (non-simple) purely infinite. Finally, in Section $4$, we  construct various interesting examples of simple and non-simple purely infinite $C^*$-algebras. In particular, in an easy way, we construct purely infinite $C^*$-algebras with torsion in their $K_1$ groups.  
  
Throughout the article we will denote by $\N$ the subsemigroup of $\Z$ given by $\{1,2,\ldots\}$ and $\Z^+$ the monoid $\N\cup \{0\}$.

\section{$\N$ and $\Z$ actions on $C^*$-algebras}

Every injective endomorhism $\alpha$ of a $C^*$-algebra $A$ can be canonically dilated to an automorphism $\bar{\alpha}$ into a bigger $C^*$-algebra $\bar{A}$, and therefore one has an intuitive way to extend the spectral theory of  automorphisms to endomorphisms. We will see later that this is the correct one since the associated crossed products are strongly Morita equivalent.  Given an endomorphism $\alpha:A\longrightarrow A$, we define the inductive system $\{A_i,\gamma_i\}_{i\in \N}$ given by $A_i:=A$ and $\gamma_i=\alpha$ for every $i\in \N$. Let $\bar{A}:=\varinjlim  \{A,\alpha\}$ the \emph{dilation of $A$ by $\alpha$}. For any $i\in\N$, $\alpha_{i,\infty}:A\longrightarrow \bar{A}$ denotes the canonical map.  The diagram 
\[
{
\def\labelstyle{\displaystyle}
 \xymatrix{ A \ar[d]^{\alpha}\ar[r]^{\alpha}  & A\ar[d]^{\alpha}\ar[r]^{\alpha} & A\ar[d]^{\alpha}\ar[r]^\alpha & \cdots\ar[r]  & \bar{A}\ar[d]^{\bar{\alpha}}\\ A \ar[r]_{\alpha}  & A\ar[r]_{\alpha} & A\ar[r]_\alpha & \cdots\ar[r] & \bar{A}  }
}\]
gives rise to an automorphism $\bar{\alpha}:\bar{A}\longrightarrow \bar{A}$, that is called \emph{the dilation of $\alpha$}. 
Recall that given an automorphism $\alpha\in \text{Aut }(A)$ we can define the \emph{spectrum of $\alpha$}, denoted by $\text{Spec }(\alpha)$, as the Gelfand spectrum of $\alpha$ viewed as the element of the Banach algebra $\mathcal{B}(A)$, bounded linear operators of $A$, and the \emph{Connes' spectrum of $\alpha$} is defined as $$\mathbb{T}(\alpha):=\bigcap_{B\in \mathcal{H}^\alpha(A)} \text{Spec }(\alpha_{|B})$$ where $ \mathcal{H}^\alpha(A)$ is the set of all the hereditary and $\alpha$-invariant sub-$C^*$-algebras $B$, i.e., $\alpha(B)=B$. Given a $C^*$-algebra $A$ we define by $\mathcal{I}(A)$ the set of the closed two-sided ideals of $A$, and given an endomorphism $\alpha\in \text{End }(A)$ we denote by $\mathcal{I}^\alpha(A)$ the set of all  $\alpha$-invariant ideals $I$ of $A$, i.e., $\alpha^{-1}(I):=\{x\in A:\alpha(x)\in I\}=I$. 

The following equivalent statements follows from Olesen and Pedersen \cite[Theorem 10.4]{OP3} and Sierakowski \cite{Sie}. These authors  studied actions of more general groups $G$, but in the situation that $G=\Z$ all simplifies in the following way: Given an automorphism $\alpha\in \text{Aut }(A)$ the following statements are equivalent:
\begin{enumerate}
\item[(1)] $\mathbb{T}(\alpha)=\mathbb{T}$,
\item[(2)] $\alpha^n$ is \emph{properly outer} \cite{El} for every $n\in \N$, i.e., $\|\alpha^n_{|I}-\text{Ad }U\|=2$ for every $\alpha$-invariant ideal $I$ of $A$ and $U$ a unitary in $M(I)$,
\item[(3)] the induced action $\Z\curvearrowright \widehat{A}$ on the space of equivalence classes of irreducible representations of $A$ is \emph{topologically free}, i.e., $\{[\pi]\in \widehat{ A}:[\pi\circ \alpha^n]=[\pi]\text{ then }n=0\}$ is dense in $\widehat{A}$,
\item[(4)] $A$ has \emph{the intersection property}, i.e., given any non-zero ideal $J$ of $A\times_\alpha \Z$ we have that $J\cap A\neq 0$,
\item[(5)] $\alpha$ satisfies the Rokhlin* property \cite{Sie} (defined below).
\end{enumerate}
Moreover, if $A$ is $\alpha$-simple, i.e., $\alpha(I)=I$ implies that $I=0,A$, then $(1)-(5)$ are equivalent to
\begin{enumerate}
\item[(6)] $\alpha^n$ is \emph{multiplier outer} for every $n\in \Z\setminus{0}$, i.e., $\alpha^n\neq \text{Ad }U$ for every unitary $U$ in the multiplier algebra $M(A)$.
\end{enumerate}

The Rokhlin* property was defined by Sierakowski \cite{Sie} for more general groups $G$. It is weaker that $G\curvearrowright \widehat{A}$ being topologically free \cite[Theorem 2.11]{Sie}, and both are equivalent when $A$ is commutative. However, if $G$ is discrete (e.g. $G=\Z$) then $\alpha$ having the Rokhlin* property implies that $\alpha$ is properly outer \cite[Theorem 2.19]{Sie}.

Observe that given any ideal $J$ of $A\times_\alpha \Z$ we have that the ideal $J\cap A$ is a $\alpha$-invariant ideal of $A$, and therefore the map $\Phi:\mathcal{I}(A\times_\alpha \Z)\longrightarrow \mathcal{I}^\alpha(A)$ defined by $J\longmapsto J\cap A$ is a surjective map. In general this map is not injective, but a necessary condition for this map being injective can be given. First observe that given any $\alpha$-ideal $I$ of $A$ we have the following short exact sequence,
$$0\longrightarrow I\times_\alpha \Z\longrightarrow A\times_\alpha \Z\longrightarrow (A/I)\times_\alpha \Z\longrightarrow 0\,,$$ 
and therefore the following statements are equivalent:
\begin{enumerate}
\item[(1)] $\mathbb{T}(\alpha_{A/I})=\mathbb{T}$ for every $\alpha$-invariant ideal $I$ of $A$,
\item[(2)] the induced action $\Z\curvearrowright\widehat{A}$ is \emph{essentially free}, i.e., the induced action $\Z\curvearrowright X$ in any $\alpha$-invariant closed subset $X$ of $\widehat{A}$ is topologically free,
\item[(3)] $A$ \emph{separates ideals of $A\times_\alpha\Z$}, i.e., $\Phi$ is injective and hence bijective,
\item[(4)] every ideal $J$ of $A\times_\alpha \Z$ is of the form $I\times_\alpha \Z$ for some $\alpha$-invariant ideal $I$,
\item[(5)] $\alpha$ satisfies the residual Rokhlin* property.
\end{enumerate} 

We recall the definition of the Rokhlin* property given by Sierakowski: Given a $C^*$-algebra $A$, set $A^\infty:=l^\infty(A)/c_0(A)$, where $l^\infty(A)$ is the $C^*$-algebra of all bounded functions from $\N$ into $A$, and $c_0(A)$ is the ideal of $l^\infty(A)$ consisting of all sequences $(a_n)$ such that $\|a_n\|\longrightarrow 0$. There is a natural inclusion $A\subseteq A^\infty$. Moreover, given an  automorphism $\alpha\in \text{Aut }(A)$, there is a natural extension to an automorphism $\alpha\in  \text{Aut }((A^\infty)^{**})$.

An automorphism $\alpha\in \text{Aut }(A)$ of a separable $C^*$-algebra $A$ has the \emph{Rokhlin* property} provided that there exists a projection  $p=(p_n)\in (A^\infty)^{**}\cap A'$, that we will call  \emph{Rokhlin projection}, such that: 
\begin{enumerate}
\item Given any $k\in \Z\setminus\{0\}$, we have that $\alpha^k(p)p=0$,
\item Given any $a\in A\setminus\{0\}$ there exists $k\in\Z$ such that $a\alpha^k(p)\neq 0$.
\end{enumerate}
If given any $\alpha$-invariant ideal $I$ of $A$ the induced automorphism $\alpha_{A/I}\in \text{Aut }(A/I)$  has the Rokhlin* property, then we say that has the \emph{residual Rokhlin* property}. 

\begin{defi} Let $A$ be  separable $C^*$-algebra. Given an endomorphism $\alpha\in \text{End }(A)$ we say that $\alpha$ satisfies the (residual) Rokhlin* property if its dilation  $\bar{\alpha}$ does.
\end{defi}

\section{The Rokhlin Property and the $p$-shift endomorphism}

Given an endomorphism $\alpha\in \text{End }(A)$, we will fix conditions on a sequence $(a_n)\in A^\infty$ to construct a Rokhlin projection for the dilated endomorphism $\bar{\alpha}\in \text{Aut }(\bar{A})$, and hence to guarantee that $\alpha$ has the Rokhlin* property.

\begin{lem}\label{lem_aux_rokhlin} Let $A$ be a separable $C^*$-algebra and let $\alpha:A\longrightarrow A$ be an injective endomorphism. 
Given  a sequence $x=(x_n)\in A^\infty$, define the sequence $y=(\alpha_{n,\infty}(x_n))\in \bar{A}^\infty$, where $\bar{A}$ is the dilation of $A$ by $\alpha$. Then:
\begin{enumerate}
\item If  for every $k\in \N$ we have that $\|x_n\alpha^k(x_n)\|\longrightarrow 0$ when $n\rightarrow\infty$, then $y\bar{\alpha}^k(y)=0$ for every $k\in\Z$.
\item If for every $l\in \N$ and $a\in A$ we have that $\|[\alpha^{n}(a),x_{l+n}]\|\longrightarrow 0$ when $n\rightarrow\infty$, then $[b,y]=0$ for every $b\in \bar{A}$.
\item If for every $l\in \N$ and $a\in A$ we have that $\|\alpha^{n}(a)x_{l+n}\|\longrightarrow \|a\|$  when $n\rightarrow\infty$, then $\|by\|=\|b\|$ for every $b\in \bar{A}$.
\end{enumerate}
\end{lem}
\begin{proof}
$(1)$ Suppose there exists a sequence $x=(x_n)\in A^\infty$ satisfying $\|x_n\alpha^k(x_n)\|\longrightarrow 0$ for every $k\in \N$ when $n\rightarrow\infty$.  
Now observe that, given $k,n\in \N$, we have that  
$$\bar{\alpha}^k(y_n)= \bar{\alpha}^k( \alpha_{n,\infty}(x_n))= \alpha_{n,\infty}(\alpha^k(x_n)) \qquad\text{and}\qquad \bar{\alpha}^{-k}(y_n)= \bar{\alpha}^{-k}( \alpha_{n,\infty}(x_n))= \alpha_{n+k,\infty}(x_n)\,.$$ 
So, given $k\in\N$ it follows that
$$\| \bar{\alpha}^k(y_n)y_n\|=\| \alpha_{n,\infty}( \alpha^k(x_n)x_n)\|=\| \alpha^k(x_n)x_n\|\longrightarrow 0$$
when $n\rightarrow\infty$, and
$$\| \bar{\alpha}^{-k}(y_n)y_n\|=\| \alpha_{n+k,\infty}(x_n \alpha^{k}(x_n))\|=\|x_n \alpha^{k}(x_n)\|\longrightarrow 0$$
when $n\rightarrow\infty$. Thus,  $y$ and $\bar{\alpha}^k(y) $ are pairwise orthogonal elements of $\bar{A}^\infty$ for every $k\in\Z\setminus\{0\}$, as we wanted.	\vspace{.2truecm}

Now, given any $b\in \bar{A}$ and $\varepsilon>0$, there exist $a\in A$ and $l\in\N$ such that $\|b-\alpha_{l,\infty}(a)\|<\varepsilon/3$. \vspace{.2truecm}

$(2)$ By hypothesis, there exists $N_a\in\N$ such that $\|\alpha^n(a)x_{l+n}-x_{l+n}\alpha^n(a)\|<\varepsilon/3$ for every $n\geq N_a$. So, we have that 
\begin{align*}
\|\alpha_{l,\infty}(a) y_{n}-y_{n}\alpha_{l,\infty}(a)\| & =\|\alpha_{n,\infty}(\alpha^{n-l}(a)x_{n}) -\alpha_{n,\infty}(x_{n}\alpha^{n-l}(a))\| \\
 & =\| \alpha^{n-l}(a)x_{n} -x_{n}\alpha^{n-l}(a)\|<\varepsilon/3\,
 \end{align*}
 for every $n\geq l+N_a$. Therefore,  we have that 
 \begin{align*}
 \|by_{n}-y_{n}b\| & \leq \|by_{n}-\alpha_{l,\infty}(a)y_{n}\|+\|\alpha_{l,\infty}(a)y_{n} - y_{n}\alpha_{l,\infty}(a)\| + \|y_{n}	\alpha_{l,\infty}(a)-y_{n}b\| \\
 & < \varepsilon/3+\varepsilon/3+\varepsilon/3=\varepsilon\,
\end{align*}
for every $n\geq l+N_a$. 
Thus, we have that $yb=by$ for every $b\in \bar{A}$.

$(3)$ By hypothesis there exists $N_a\in\N$, such that $\|\alpha^n(a)x_{l+n}\|\geq \|a\|-\varepsilon/3$ for every $n\geq N_a$. So we have that 
$$\|\alpha_{l,\infty}(a) y_{n}\|  =\|\alpha_{n,\infty}(\alpha^{n-l}(a)x_{n})\|
 =\|\alpha^{n-l}(a)x_{n}\|\geq \|a\|-\varepsilon/3\,,$$
for every $n\geq l+N_a$. Therefore,  we have that 
$$\|\alpha_{l,\infty}(a)y_{n}\|\leq \|\alpha_{l,\infty}(a) y_{n}-b y_{n}\|+\|b y_{n}\|<\varepsilon/3+\|by_{n}\| ,$$
for every $n\geq l+N_a$, and thus
$$\|\alpha_{l,\infty}(a)\|-\varepsilon<\|\alpha_{l,\infty}(a)\|-\varepsilon/3-\varepsilon/3\leq \|\alpha_{l,\infty}(a) y_{n}\|-\varepsilon/3<\|by_{n}\|$$
for every $n\geq l+N_a$. But since $\|b\|-\varepsilon\leq\|a\|=\|\alpha_{l,\infty}(a)\|$, it follows that 
$$\|b\|-2\varepsilon\leq \|by_{n}\|\,$$
for every $n\geq l+N_a$. As $\varepsilon$ is arbitrary, we have that $\|b\|=\|by\|$, as desired.
\end{proof}

Observe that, given a projection $x=(x_n)\in A^\infty$ satisfying the hypothesis $(1)-(3)$ of the above Lemma, the sequence $p=(\alpha_{n,\infty}(x_n))\in \bar{A}^\infty$ is a Rokhlin projection for the automorphism $\bar{\alpha}$. Observe that then $\alpha$ satisfies even a stronger property than the Rokhlin* property, because $\|pb\|=\|b\|$ for every $b\in \bar{A}$. 

Given a unital simple $C^*$-algebra $A$ with a non-trivial projection $p$, Dykema and R{\o}rdam defined in \cite{DR} the so called \emph{Bernoulli $p$-shift} endomorphism $\Delta_p: A^{\otimes \infty}\longrightarrow A^{\otimes \infty}$ by $x\longmapsto p\otimes x$ for every $x\in A^{\otimes \infty}$. Choosing $\otimes=\otimes_{max}$ or $\otimes_{min}$ we have that $\Delta_p$ is injective. They proved that any power of the dilation automorphism $\bar{\Delta_p}$ is multiplier outer. Hence, since $A^{\otimes \infty}$ is simple, $\bar{\Delta_p}$ must be properly outer.  We will prove a more general result.

\begin{prop}\label{lem_rokhlin} Let $A$ be a nuclear unital $C^*$-algebra with a non-trivial projection $p$, and let $\Delta_p:A^{\otimes \infty}\longrightarrow A^{\otimes \infty}$ be the $p$-shift endomorphism. Then, $\Delta_p$ satisfies the Rokhlin* property. 
\end{prop}
\begin{proof} To simplify notation set $B:=A^{\otimes \infty}$ and $\alpha:=\Delta_p$. Given $n\in\N$, we define the projection 
$$q_n:=1\otimes \cdots^{(2n)}\cdots \otimes1\otimes (1-p)\otimes p\otimes \cdots^{(2n)}\cdots\otimes p\otimes 1_{A^{\otimes \infty}}\in B\,.$$

Observe that $q=(q_n)\in B^\infty$ is a projection, and then by Lemma \ref{lem_aux_rokhlin} it is enough to check the following: 
\begin{enumerate}
\item $\|q_n\alpha^k(q_n)\|\longrightarrow 0$ when $n\rightarrow\infty$, for every $k\in \N$,
\item $\|[\alpha^{n}(b),q_{l+n}]\|\longrightarrow 0$ when $n\rightarrow\infty$, for every $l\in \N$ and $b\in B$, 
\item $\|\alpha^{n}(b)q_{l+n}\|\longrightarrow \|b\|$  when $n\rightarrow\infty$, for every $l\in \N$ and $b\in B$,
\end{enumerate}
to prove that $q'=(\alpha_{n,\infty}(q_n))  \in \bar{B}^\infty$ is a Rokhlin projection for $\bar{\alpha}$.

First, we have that $q_n$ and $\alpha^k(q_n)$ are orthogonal projections for every $k\in\N$, so $(1)$ holds. 

Now given any $b\in B$ and $\varepsilon>0$, there exists $j\in \N$ and  $a\in A^{\otimes j}$ such that $\|b-a\otimes 1_{A^{\otimes \infty}}\|<\varepsilon/2$. Therefore, given any $l,n\in \N$ we have that 
$$\alpha^n(a\otimes 1_{A^{\otimes \infty}})=p\otimes \cdots^{(n)}\cdots\otimes p\otimes a\otimes 1_{A^{\otimes \infty}}\,,$$
and
$$q_{l+n}=1\otimes \cdots^{(2(l+n))}\cdots \otimes1\otimes (1-p)\otimes p\otimes \cdots^{(2(l+n))}\cdots\otimes p\otimes 1_{A^{\otimes \infty}}\in B\,.$$
Thus, whenever $n\geq j-2l$, it follows that 
$$(*)\qquad q_{l+n}\alpha^n(a\otimes 1_{A^{\otimes \infty}})=\alpha^n(a\otimes 1_{A^{\otimes \infty}})q_{l+n}=\alpha^n(a\otimes 1_{A^{\otimes \infty}})q_{l+n}=p\otimes \cdots^{(n)}\cdots \otimes p\otimes a\otimes z_{l,n}$$
where $z_{l,n}$ is a projection of $A^{\otimes \infty}$. Therefore, we have that
\begin{align*}\|\alpha^n(b)q_{l+n}-q_{l+n}\alpha^n(b)\| & \leq \|\alpha^n(b)q_{l+n}-\alpha^n(a\otimes 1_{A^{\otimes \infty}})q_{l+n}\| \\ & +\|\alpha^n(a\otimes 1_{A^{\otimes \infty}})q_{l+n}-q_{l+n}\alpha^n(a\otimes 1_{A^{\otimes \infty}})\| \\ & +\|q_{l+n}\alpha^n(a\otimes 1_{A^{\otimes \infty}})-q_{l+n}\alpha^n(b)\| \\
 & <\varepsilon/2+0+\varepsilon/2=\varepsilon\,
\end{align*}
for every $n\geq j-2l$. Then $\|[\alpha^{n}(b),q_{l+n}]\|\longrightarrow 0$ when $n\rightarrow\infty$, so $(2)$ holds.

 Finally, observe that using the canonical isomorphism $C^{\otimes \infty}\cong C^{\otimes n}\otimes C^{\otimes j}\otimes C^{\otimes \infty}$, and applying  $(*)$  and the nuclearity of $A$ it follows that 
$$\|\alpha^n(a\otimes 1_{A^{\otimes \infty}})q_{l+n}\|=\|a\|$$
for every $n\geq j-2l$. So, we have that
$$\|\alpha^n(b)\|-\varepsilon< \|\alpha^n(a\otimes 1_{A^{\otimes \infty}})\|-\varepsilon/2=\|\alpha^n(a\otimes 1_{A^{\otimes \infty}})q_{l+n}\|-\varepsilon/2<\|\alpha^n(b)q_{l+n}\|$$
for every $n\geq j-2l$. Then, $q'=(q_n)$ also satisfies $(3)$.

Thus, by Lemma \ref{lem_aux_rokhlin} $p=(\alpha_{n,\infty}(q_n))\in \bar{B}^\infty$ is a Rokhlin projection for $\overline{\Delta_p}$.
\end{proof}

\begin{exem}
Given any $n\in\N$ we set $p:=e_{1,1}\in M_n(\mathbb{C})$. Then, the $p$-shift endomorphism $\Delta_{p}$ defined on the UHF-algebra $M_{n^\infty}:=\bigotimes_{i=1}^\infty M_n(\mathbb{C})$ is the Cuntz's endomorphism. By Proposition \ref{lem_rokhlin}, the automorphism $\bar{\Delta_p}:M_{n^\infty}\otimes\mathcal{K}\longrightarrow M_{n^\infty}\otimes\mathcal{K}$ satisfies the  Rokhlin* property.
\end{exem}

\begin{exem} Let $A=\C\oplus \C$ and let $p:=(1,0)$. Then we have that $A^{\otimes \infty}\cong C(X)$, where $X=\{0,1\}^{\N}$ is the Cantor set. In this case, the $p$-shift endomorphism $\Delta_p:C(X)\longrightarrow C(X)$  is defined as $\Delta_p(f)=f\circ \delta$ for every $f\in C(X)$, where $\delta(\textbf{x})=(0,\textbf{x})$ for every $\textbf{x}\in X$. By Proposition \ref{lem_rokhlin}, $\bar{\Delta_p}$ is properly outer. Let us consider the ideal $I=\{f\in C(X):f(\textbf{0})=0\}$, that is a $\Delta_p$-invariant ideal. Then $C(X)/I\cong \C$ and $(\Delta_p)_{C(X)/I}=\overline{(\Delta_p)_{C(X)/I}}=\text{Id}$, so it is not properly outer.
\end{exem}

Now, we will study the $p$-shift endomorphism on an interesting class of $C^*$-algebras.

Let $\D$ be a unital and separable \emph{strongly self-absorbing} $C^*$-algebra \cite{TW}. Then, $\D\cong \D^{\otimes n}\cong \D^{\otimes \infty}$ for every $n\in \N$. Also, $\D$ has an approximately inner flip, i.e. the homomorphism $\sigma:\D\otimes \D\longrightarrow \D\otimes \D$ defined by $\sigma(a\otimes b)=b\otimes a$ for every $a,b\in \D$ is approximately unitary equivalent to the identity map. Recall also that every strongly self-absorbing $C^*$-algebra is nuclear.  

We will first compute which map induces $\Delta_p$ at the level of $K$-theory. Recall that a strongly self-absorbing $C^*$-algebra $\D$ satisfying the Universal Coefficients Theorem (UCT) has trivial $K_1$ group \cite[Proposition 5.1]{TW} and torsion-free $K_0$ group. Therefore, the K\"unneth formulas say that we have an isomorphism $K_0(\D^{\otimes \infty})=K_0(\D\otimes \D^{\otimes \infty})\cong K_0(\D)\otimes K_0(\D^{\otimes \infty})$ induced by the map $[a]\otimes [b]\longmapsto [a\otimes b]$ for projections $a\in \D$, $b\in \D^{\otimes\infty}$. Thus, if we define $[p]\cdot x:= [p]\otimes x$ for every $x\in K_0(\D^{\otimes \infty})$, the following Lemma comes straightforward.

\begin{lem}\label{1.5} Let $\D$ be a strongly self-absorbing $C^*$-algebra satisfying the UCT with a non-trivial projection $p$. Then $K_0(\Delta_p)=[p]\cdot \text{Id}_{K_0(\D^{\otimes \infty})}$.
\end{lem}

We will prove some results about ideal structure for suitable tensor products of $C^*$-algebras. We thank Nate Brown for the proof of the following result.

\begin{lem}\label{Nate}
Let $A$ be an exact simple $C^*$-algebra, and let $B$ be any $C^*$-algebra. If $I\lhd A\otimes B$, then there exists $J\lhd B$ such that $I=A\otimes J$.
\end{lem}
\begin{proof}
By \cite[Corollary 9.4.6]{B-O}, 
$$I=\overline{\text{span}}\{ A\odot I_B : I_B\lhd B, A\odot I_B\subseteq I\}.$$
Consider $$J:=\overline{\text{span}}\{ I_B\lhd B : A\odot I_B\subseteq I\},$$ 
which is clearly an ideal of $B$. Now, pick 
$$X:=A\odot \left({\text{span}}\{ I_B : I_B\lhd B, A\odot I_B\subseteq I\}\right),$$ 
which is a dense linear subspace of $A\odot J$. We will show that $X$ is also a dense linear subspace of $I$. For, notice that
$$ {\text{span}}\{ A\odot I_B : I_B\lhd B, A\odot I_B\subseteq I\}=A\odot \left({\text{span}}\{ I_B : I_B\lhd B, A\odot I_B\subseteq I\}\right).$$
Since closures are unique, we conclude that $I=A\otimes J$, as desired.
\end{proof}

As a consequence, we have the following $C^*$-algebra version of Azumaya-Nakayama's Theorem.

\begin{prop}\label{Azumaya-Nakayama}
If $A$ is an exact simple $C^*$-algebra and $B$ is any $C^*$-algebra, then the map $I\mapsto A\otimes I$ defines a bijection between ideals of $B$ and ideals of $A\otimes B$
\end{prop}
\begin{proof}
Clearly, it is a well-defined map. By Lemma \ref{Nate}, it is a surjective map. In order to prove injectivity of the map, notice that any ideal can be seen as a linear subspace of $A\otimes B$ via the identification $I=\C \otimes I=\C\odot I$.

By \cite[Proposition 4.7.3]{Cohn}, if $I\lhd B$, then 
$$(A\odot I)\cap B=(A\odot I)\cap(\C \odot I)=\C \odot I=I.$$
Thus, taking closures we have
$$(A\otimes I)\cap B=(A\otimes I)\cap(\C \otimes I)=\C \otimes I=I,$$
whence the above defined map is injective. So we are done.
\end{proof}

We say that a $C^*$-algebra $A$ \emph{absorbs $\D$} if  $A\cong A\otimes \D$. 

\begin{theor}\label{rokhlin_ideals} Let $\D$ be a strongly self-absorbing $C^*$-algebra with a non-trivial projection $p$, and let $A$ be a separable unital  $C^*$-algebra that absorbs $\D$. Then, given any injective endomorphism $\alpha\in \text{End }(A)$, the  endomorphism $\alpha\otimes \Delta_p\in \text{End }(A\otimes \D^{\otimes \infty})$   satisfies the residual Rokhlin* property. Moreover, if $\D$ satisfies the UCT then $K_*(\alpha\otimes \Delta_p)=[p]\cdot K_*(\alpha)$.
\end{theor}
\begin{proof}
Let $\alpha\in\text{End }(A)$ be  an injective endomorphism and let $\Delta_p:\D^{\otimes\infty}\longrightarrow \D^{\otimes\infty}$ be the $p$-shift endomorphism. By Proposition \ref{Azumaya-Nakayama}, all the ideals of  $A\otimes \D^{\otimes \infty}$ are of the form $I\otimes \D^{\otimes \infty}$ for ideals $I$ of $A$. Hence, an ideal $I\otimes \D^{\otimes \infty}$ is $\alpha\otimes \Delta_p$-invariant ideal if and only if  $I$ is an $\alpha$-invariant ideal of $A$. 

Given any $\alpha$-invariant ideal $I$ of $A$, let us consider the quotient $B:=(A\otimes \D^{\otimes \infty})/(I\otimes \D^{\otimes \infty})$. Observe that we can identify $B$ with $(A/I)\otimes \D^{\otimes \infty}$, as $\overline{a\otimes x}\longmapsto \overline{a}\otimes x$ for every $a\in A$ and $x\in\D^{\otimes \infty}$.  Moreover, with this identification, we have that $(\alpha\otimes\Delta_p)_B=\alpha_{A/I}\otimes \Delta_p$, and to simplify notation set $\beta:=\alpha_{A/I}\otimes \Delta_p$.

Let us consider the projection $q=(q_n)\in (\D^{\otimes \infty})^\infty$ defined in the proof of Proposition \ref{lem_rokhlin}. If $1_A$ is the unit of $A$, then we define the non-zero projection  $q'=(\beta_{n,\infty}(\overline{1_A}\otimes q_n))\in \bar{B}^\infty$, where $\bar{B}$ is the dilation of $B$ with respect to $\beta$. We claim that $q'$ is a Rokhlin projection for $\bar{\beta}$.  For this it is enough to check conditions $(1)-(3)$ of  Lemma \ref{lem_aux_rokhlin}. Since $\beta^k(\overline{1_A}\otimes q_n)=\overline{\alpha(1_A)}\otimes \Delta_p^k(q_n)$  for every $k,n\in\N$, it follows that $$\beta^k(\overline{1_A}\otimes q_n)\cdot (\overline{1_A}\otimes q_n)=\overline{\alpha(1_A)}\otimes (\Delta_p^k(q_n)\cdot q_n)=0,$$
since $\Delta_p^k(q_n)\cdot q_n =0$ for every $k\in\N$. Thus, $(1)$ is checked.

Now, given $a\in (A/I)\otimes \D^{\otimes \infty}$ and $\varepsilon>0$ there exist $r,s\in \N$ such that $b:=\sum_{i=1}^r\overline{a_i}\otimes (x_i\otimes 1_{\D^{\otimes \infty}})$ for some  $a_i\in A$ and $x_i\in \D^{\otimes s}$ and such that $\|a-b\|<\varepsilon/2$. Then, as we see in the proof of Proposition  \ref{lem_rokhlin}, given $l\in \N$ for every $n\geq s-2l$ we have that 
$$\Delta^n_p(x_i\otimes 1_{\D^{\otimes \infty}} )q_{l+n}=q_{l+n}\Delta^n_p(x_i\otimes 1_{\D^{\otimes \infty}} )=p\otimes \cdots^{(n)}\cdots \otimes p\otimes x_i\otimes z_{l,n}$$
   for every $i=1,\ldots,r$ and for some  projection $z_{l,n}\in \D^{\otimes \infty}$. Therefore, a standard argument shows that  $$\|\beta^n(a)(\overline{1_A}\otimes q_{l+n})-(\overline{1_A}\otimes q_{l+n})\beta^n(a)\|<\varepsilon$$	
for every $n\geq s-2l$, and then $\|[\beta^{n}(a),q_{l+n}]\|\longrightarrow 0$ when $n\rightarrow\infty$. Therefore it follows $(2)$.

Finally,  we have that
\begin{align*} \|\beta^n(\sum_{i=1}^r\overline{a_i}\otimes (x_i\otimes 1_{\D^{\otimes \infty}}))(\overline{1_A}\otimes q_{l+n})\| & =\|\sum_{i=1}^r\overline{\alpha^n(a_i)}\otimes (\Delta^n_p(x_i\otimes 1_{\D^{\otimes \infty}})q_{l+n}) \| \\ & =\|\sum_{i=1}^r\overline{\alpha^n(a_i)}\otimes (p\otimes \cdots^{(n)}\cdots \otimes p\otimes x_i\otimes z_{l,n})\|\,,
\end{align*}
and using the canonical isomorphism $A\otimes \D^{\otimes \infty}\cong A\otimes \D^{\otimes s} \otimes \D^{\otimes \infty}$ and the nuclearity of $\D$ it follows that
\begin{align*} \|\sum_{i=1}^r\overline{\alpha^n(a_i)}\otimes (p\otimes \cdots^{(n)}\cdots \otimes p\otimes x_i\otimes z_{l,n})\| & =\|\sum_{i=1}^r\overline{\alpha^n(a_i)}\otimes x_i\| \|p\otimes \cdots^{(n)}\cdots \otimes p\otimes z_{l,n}\| \\
& =\|\sum_{i=1}^r\overline{\alpha^n(a_i)}\otimes  x_i\|=\|\sum_{i=1}^r\overline{a_i}\otimes  x_i\|\,,
\end{align*}
for every $n\geq s-2l$. Another standard argument shows us that 
$$\|a\|-\varepsilon=\|\beta^n(a)\|-\varepsilon<\|\beta^n(a)q_{l+n}\|\,,$$
for every $n\geq s-2l$. Thus,  $\|\beta^n(a)q_{l+n}\|\longrightarrow \|a\|$ when $n\rightarrow\infty$, and then condition $(3)$ is verified. Therefore $q'$ is a Rokhlin projection for $\bar{\beta}$, and hence $\beta$ satisfies the Rokhlin* property. Since $B$ denotes $(A\otimes \D^{\otimes \infty})/(I\otimes \D^{\otimes \infty})$ for every arbitrary $\alpha$-invariant ideal $I$ of $A$, we have that $\alpha\otimes \Delta_p$ satisfies the residual Rokhlin* property. 

Finally if $\D$ satisfies the UCT then $K_1(\D)=0$  and $K_0(\D)$ is torsion-free, and hence by the K\"unneth formulas we have that $K_*(A\otimes \D^{\otimes \infty})\cong K_*(A)\otimes K_*(\D^{\otimes \infty})$ and $K_*(\alpha\otimes \Delta_p)=K_*(\alpha)\otimes K_*(\Delta_p)$, so $K_*(\alpha\otimes \Delta_p)=[p]\cdot K_*(\alpha)$ by Lemma \ref{1.5}.

\end{proof}

Notice that if $A$ absorbs $\D$, then by \cite[Theorem 2.3]{TW} any isomorphism $\varphi:A\longmapsto A\otimes\D$  is approximately unitary equivalent to  $\text{Id}_A\otimes 1_\D$.

\begin{lem}\label{A-N per O infinit}
If $\D$ is a strongly self-absorbing $C^*$-algebra, $A$ is a $C^*$-algebra that absorbs $\D$, and $\varphi : A\rightarrow A\otimes \D$ is an isomorphism, then for any $I\lhd A$ we have that $\varphi (I)\subseteq I\otimes \D$. Moreover, if $A$ has finitely many ideals, then $\varphi (I)= I\otimes \D$.
\end{lem}
\begin{proof}
By \cite[Theorem 2.3]{TW}, $\varphi$ is approximately unitary equivalent to $\text{Id}_A\otimes 1_{\D}$. Now, as
$$(\text{Id}_A\otimes 1_{\D})(I)\subseteq I\otimes \D,$$
given any unitary $u\in M(A\otimes \D)$ we have $u(\text{Id}_A\otimes 1_{\D})(I)u^*\subseteq I\otimes \D$. Hence, $\varphi (I)\subseteq I\otimes \D$.

For the last statement, notice that by the previous statement and Proposition \ref{Azumaya-Nakayama}, for each $I\lhd A$ there exists a unique $J_I\lhd A$ such that 
$$\varphi (I)=J_I\otimes \D \subseteq I\otimes \D,$$
and thus $J_I\subseteq I$ again by Proposition \ref{Azumaya-Nakayama}. Hence, there is a monotone decreasing bijection $(-)_I$ from the set of ideals of $A$ to itself. Notice that $0$ and $A$ are fixed by $(-)_I$. If $(-)_I$ is not the identity map, as the set of ideals is finite, there exists $K\lhd A$ different from $0$ and $A$ maximal for the property $K_I\ne K$. By the above argument there exists a unique $L\lhd A$ such that $L_I=K$, and notice that $K\subseteq L$. Assuming $K\ne L$ will contradict the maximality of $K$. So, $K=L$, but then
$K=L_I=K_I\subsetneq K$, which is impossible. So we are done.
\end{proof}

\begin{lem}\label{prop_Rokhlin} Let $B,C$ be $C^*$-algebras, let $\beta:B\longrightarrow B$ be an  injective endomorphism, and let $\delta:C\longrightarrow B$ be an isomorphism. If $\beta$  satisfies the residual Rokhlin* property, then so does $\delta^{-1}\circ\beta\circ \delta$.
\end{lem}
\begin{proof} First define $\alpha:=\delta^{-1}\circ\beta\circ \delta$, so it is easy to check that the dilated automorphism $\bar{\alpha}=\overline{\delta^{-1}\circ\beta\circ \delta}$ of $\bar{C}$ is the composition $\bar{\delta}^{-1}\circ\bar{\beta}\circ \bar{\delta}$, where $\bar{\delta}$ is the induced isomorphism between $\bar{C}$ and $\bar{B}$. Observe also that  $\mathcal{I}^{\alpha}(C)=\{\delta^{-1}(I):I\in \mathcal{I}^\beta(B)\}$, and then given any $I\in \mathcal{I}^\beta(B)$ we have that $\alpha_{A/\delta^{-1}(I)}=\delta^{-1}_{C/\delta^{-1}(I)}\circ\beta_{B/I}\circ \delta_{C/\delta^{-1}(I)}$

Finally, let $p\in (\bar{B}^\infty)^{**}\cap \bar{B}'$ be a Rokhlin  projection of $\beta$, then since $\delta$ induces an isomorphism $\delta:(\bar{C}^\infty)^{**}\longrightarrow (\bar{B}^\infty)^{**}$ we have that $\delta^{-1}(p)$ is a Rokhlin projection for $\alpha$.
\end{proof}

Recall that if $\D$ is  a strongly self-absorbing $C^*$-algebra in the UCT class and $A$ is a $C^*$-algebra that absorbs $\D$, then the isomorphism  $\varphi:A \longrightarrow  A\otimes \D$  induces the isomorphism of $K$-theory $K_*(\varphi):K_*(A)\longrightarrow K_*(A)\otimes K_0(\D)$ given by $x\longmapsto x\otimes [1_{\D}]$ for every $x\in K_*(A)$. We define $[q]\cdot x:=K_*(\varphi)^{-1}(x\otimes [q])$ for every projection $q\in \D$.

\begin{corol}\label{corol_rokhlin} Let $\D$ be a strongly self-absorbing $C^*$-algebra with a non-trivial projection $p$, and let $A$ be a separable unital  $C^*$-algebra with finitely many ideals such that $A$ absorbs $\D$. Then, given any injective endomorphism $\alpha\in \text{End }(A)$ there exists an  endomorphism $\beta\in \text{End }(A)$ such that $\mathcal{I}^{\alpha}(A)=\mathcal{I}^{\beta}(A)$  and satisfies the residual Rokhlin* property. Moreover, if $\D$ satisfies the UCT then $K_*(\beta)=[p]\cdot K_*(\alpha)$.
\end{corol}
\begin{proof} We define $\beta:=\varphi^{-1}\circ (\alpha\otimes \Delta_p)\circ \varphi$ where $\varphi$ is any isomorphism $A\cong A\otimes \D^{\otimes \infty}$. First observe that that by Lemma \ref{A-N per O infinit} we have that $\mathcal{I}^{\beta}(A)=\{\varphi^{-1}(I\otimes \D^{\otimes \infty} ):I\in \mathcal{I}^\alpha(A)\}=\mathcal{I}^{\alpha}(A)$. Now by Theorem \ref{rokhlin_ideals} and Lemma \ref{prop_Rokhlin} we have that $\beta$ has the residual Rokhlin* property.

If $\D$ satisfies the UCT we have that $K_*(\beta)=K_*(\varphi^{-1})\circ K_*(\alpha\otimes \Delta_p)\circ K_*(\varphi)$. Therefore, 
\begin{align*}K_*(\beta)(x) & =K_*(\varphi^{-1})\circ K_*(\alpha\otimes \Delta_p)\circ K_*(\varphi)(x) \\ & =K_*(\varphi^{-1})\circ K_*(\alpha\otimes \Delta_p)(x\otimes [1_{\D^{\otimes \infty}}]) \\
& = K_*(\varphi^{-1})(K_*(\alpha)(x)\otimes [p\otimes 1_{\D^{\otimes \infty}}])=[p\otimes 1_{\D^{\otimes \infty}}]\cdot K_*(\alpha)(x)
\end{align*}
for every $x\in K_*(A)$. Finally, since  the map $[q]\longmapsto [q\otimes 1_{\D^{\otimes \infty}}]$ for $[q]\in K_0(\D)$ induces an isomorphism between $K_0(\D)$ and $K_0(\D^{\otimes\infty})$ we have that $[p\otimes 1_{\D^{\otimes \infty}}]\cdot K_*(\alpha)(x)=[p]\cdot K_*(\alpha)(x)$ for every $x\in K_*(A)$.
\end{proof}

\section{Purely infinite crossed products}

Given a unital $C^*$-algebra $A$ and an injective endomorphism $\alpha:A\longrightarrow A$ Stacey \cite{Stacey}, following the ideas of Cuntz \cite{Cu} and Paschke \cite{Pa}, defined the crossed product $A\times_\alpha\N$ as the universal $C^*$-algebra generated by $A$ and an  isometry $S_\infty$ such that $S_\infty a S_\infty^*=\alpha(a)$ for every $a\in A$. It was also shown that $A\times_\alpha\N\cong \alpha_{1,\infty}(1_A)(\bar{A}\times_{\bar{\alpha}}\Z)\alpha_{1,\infty}(1_A)$ and that  $\alpha_{1,\infty}(1_A)$ is a full projection. Therefore $A\times_\alpha\N$ and $\bar{A}\times_{\bar{\alpha}}\Z$ are strongly Morita equivalent. So, to study properties like purely infiniteness, simplicity or the ideal structure of $A\times_\alpha \N$, it is enough to look at the crossed product $\bar{A}\times_{\bar{\alpha}}\Z$. 

From now on, we will identify $A\times_\alpha\N$ with the corresponding isomorphic corner sub-$C^*$-algebra of  $\bar{A}\times_{\bar{\alpha}}\Z$, while $A$ is identified with $\alpha_{1,\infty}(A)$ and the generating isometry $S_\infty$ of $A\times_\alpha\N$ with the compression $\alpha_{1,\infty}(1)U_\infty \alpha_{1,\infty}(1)$ of the generating unitary $U_\infty$ of $\bar{A}\times_{\bar{\alpha}}\Z$.

There is a bijection $\lambda:\mathcal{I}^\alpha(A)\longrightarrow\mathcal{I}^{\bar{\alpha}}(\bar{A})$ given by $I\longmapsto \sum_{i=1}^\infty\alpha_{i,\infty}(I)$,  with inverse $\lambda^{-1}(J)=\{a\in A: \alpha_{1,\infty}(a)\in J\}$ for every $J\in \mathcal{I}^{\bar{\alpha}}(\bar{A})$, that makes commutative the diagram 
\[
{
\def\labelstyle{\displaystyle}
 \xymatrix{ {\mathcal{I}(A\times_\alpha \N)}\ar[r]^{\Phi}\ar[d]^{ \mu}  & {\mathcal{I}^\alpha(A)}\ar[d]^{\lambda} \\ {\mathcal{I}(\bar{A}\times_{\bar{\alpha}}\Z)} \ar[r]^{\Phi}  & \mathcal{I}^{\bar{\alpha}}(\bar{A})   }
}\]
where $\mu$ is the bijection induced by the Morita equivalence.  Hence, $\Phi:\mathcal{I}(\bar{A}\times_{\bar{\alpha}}\Z)\longrightarrow \mathcal{I}^{\bar{\alpha}}(\bar{A})$ is a bijection if and only if so is $\Phi:\mathcal{I}(A\times_{\alpha}\N)\longrightarrow \mathcal{I}^{\alpha}(A)$. Thus,  there is no ambiguity in saying that $A$ separates ideals of $A\times_\alpha\N$ exactly when $\bar{A}$ separates ideals of $\bar{A}\times_{\bar{\alpha}}\Z$. Now, given $I\in \mathcal{I}^{\alpha}(A)$  we have that  $\alpha$ restricts to an endomorphism of $I$, and therefore we can identify $\sum_{i=1}^\infty\alpha_{i,\infty}(I)$ with $\bar{I}$, the dilation of $I$ by $\alpha_{|I}$.  Since $\bar{I}$ is an $\bar{\alpha}$-invariant ideal of $\bar{A}$, we can identify $\bar{I}\times_{\bar{\alpha_{|I}}} \Z$ with an ideal of $\bar{A}\times_{\bar{\alpha}} \Z$. Moreover,  $\bar{A}/\bar{I}$ is isomorphic to $\overline{A/I}$, the dilation of $A/I$ by $\alpha_{A/I}$, where $\alpha_{A/I}$ is the natural endomorphism induced by the quotient. Thus, given $I\in \mathcal{I}^{\alpha}(A)$ we have the following short exact sequence: 
$$0\longrightarrow \bar{I}\times_{\bar{\alpha_{|I}}} \Z\longrightarrow\bar{A}\times_{\bar{\alpha}} \Z\longrightarrow \overline{A/I}\times_{\overline{\alpha_{A/I}}} \Z\longrightarrow 0\,.$$

Given $I\in \mathcal{I}^{\alpha}(A)$,  the ideal $\langle I\rangle$ of $A\times_\alpha\N$ generated by $I$ is not necessarily isomorphic to $I\times_{\alpha_{|I}} \N$, as it is defined for the non-unital case, but it is isomorphic to $\alpha_{1,\infty}(1_A) (\bar{I}\times_{\bar{\alpha_{|I}}}\Z)\alpha_{1,\infty}(1_A)$. Indeed,  to simplify notation let us denote $B:=\bar{A}\times_{\bar{\alpha}}\Z$ and $p:=\alpha_{1,\infty}(1_A)\in B$, so that by the above identification we have that $pBp:=A\times_\alpha\N$. From the observation that $U^*_\infty \alpha_{1,n}(y)U_\infty=\alpha_{1,n+1}(y)$ for every $n\in \N$ and $y\in I$, it follows that $BIB=B\bar{I}B$, and hence 
$$\langle I\rangle=pBpIpBp=pBIBp=pB\bar{I}Bp\,.$$
But $B\bar{I}B$ is $\bar{I}\times_{\bar{\alpha_{|I}}}\Z$, so the claim is proved. Therefore we have that  $\langle I\rangle$ is strongly Morita equivalent to $\bar{I}\times_{\bar{\alpha}}\Z$, so $K_*(\langle I\rangle)\cong K_*(\bar{I}\times_{\bar{\alpha_{|I}}}\Z)$, and we can use the Pimsner-Voiculescu six-terms exact sequence, together with the continuity of the $K$-theory, to compute the $K$-groups of the ideals of the form $\langle I\rangle$ for $I\in \mathcal{I}^{\alpha}(A)$. Moreover, given $I,J\in \mathcal{I}^{\alpha}(A)$ with $I\subseteq J$, by the previous argument we have that 
$$\langle J\rangle/\langle I\rangle=\alpha_{1,\infty}(1) (\overline{J/I}\times_{\overline{\alpha_{J/I}}}\Z)\alpha_{1,\infty}(1)\,.$$

\begin{rema} We would like to remark that our definition of invariant ideal slightly differs from the one given by Adji in \cite{Adji} for two reasons. First, because we only are interested in actions by injective endomorphisms. And second, because we are not interested  for a characterization of the gauge invariant ideals as another crossed product.

\end{rema}

Our goal in this section is to give conditions for the crossed product $A\times_\alpha\N$ being purely infinite. 

Notice that, if $A$ is a separable purely infinite $C^*$-algebra of real rank zero and $\alpha\in \text{End }(A)$ is such that it satisfies the residual Rokhlin* property, then the dilation $\bar{A}$ is a separable purely infinite $C^*$-algebra of real rank zero and  the action of $\Z$ induced by $\bar{\alpha}$ on $\widehat{(\bar{A})}$ is essentially free. Hence, $\bar{A}\times_{\bar{\alpha}}\Z$ is a purely infinite $C^*$-algebra \cite[Theorem 3.3]{RS}, and thus so is $A\times_\alpha \N$. 

Another condition we will consider is related to endomorphisms on $C^*$-algebras of real rank zero that acts in  a special way at the level of their monoid of projections. Given a $C^*$-algebra $A$, let $V(A)$ be the monoid of Murray-von Neumann equivalence classes of projections of $M_\infty(A)$. We briefly recall its construction below.

We say that two projections $p$ and $q$ in a $C^*$-algebra $A$ are Murray-von Neumann \emph{equivalent} if there is a partial isometry $v$ such that $p=vv^*$ and $v^*v=q$. One can extend this relation to the set $\mathcal{P}(M_{\infty}(A))$ of projections in $M_{\infty}(A)$, and thereby construct
\[
V(A)=\mathcal{P}(M_{\infty}(A))/\!\!\sim\,,
\]
where $[p]\in V(A)$ stands for the equivalence class that contains the projection $p$ in $M_{\infty}(A)$.
This set becomes an abelian monoid when endowed with the operation $[p]+[q]=[p\oplus q]$, where $p\oplus q$ refers to the matrix $\left(\begin{smallmatrix} p & 0 \\ 0 & q\end{smallmatrix}\right)$. Moreover, if $I$ is a closed two sided ideal of $A$, then $V(I)$ is an ideal of $V(A)$. If $A$ has real rank zero, then all the ideals of $V(A)$ are of the form $V(I)$ for some ideal $I$ of $A$, and moreover $V(A)/V(I)\cong V(A/I)$. Now, given any $C^*$-algebra endomorphism $\alpha:A\longrightarrow A$, we can extend it to a endomorphism $\alpha:M_{\infty}(A)\longrightarrow M_\infty(A)$ by  $a\otimes k\mapsto \alpha(a)\otimes k$ for every $a\in A$ and $k\in M_{\infty}(\C)$. Hence, it induces a map $\alpha^*:V(A)\longrightarrow V(A)$ by the rule $[p]\longmapsto [\alpha(p)]$ for every projection $p\in M_\infty(A)$.
 
\begin{defi} Let $A$ be a $C^*$-algebra and let $\alpha\in \text{End }(A)$. Then we say that $\alpha$ \emph{contracts projections of $A$} if given $x\in V(A)$ there exists $n\in \N$ such that $\alpha^{*n}(x)<x$.  
\end{defi}
\begin{exem}
Given a $C^*$-algebra $A$ and a strongly self-absorbing $C^*$-algebra $\D$ with a non-trivial projection $p$, we have that $A\otimes \D^{\otimes \infty}=\overline{\bigcup_{n=1}^\infty A_n}$ where $A_n:=A\otimes\D^{\otimes n}\otimes 1_{\D^{\otimes \infty}}$. Therefore given any $q\in A\otimes \D^{\otimes \infty}$ there exists $q'\in A\otimes\D^{\otimes n}$ for some $n\in \N$ such that $q\sim q'\otimes 1_{\D^{\otimes \infty}}$. Since $\D$ is strongly self-absorbing, the flip $\sigma:\D\otimes \D\longrightarrow \D\otimes \D$ is approximately unitary equivalent to $\text{Id}_{\D\otimes \D}$. Hence, if we write $q'=\sum_i a_i\otimes q_i'$ where $a_i\in A$ and $q_i\in \D^{\otimes n}$, then there exists a unitary $u\in \D^{\otimes n+1}$ such that 
$$\|(1_A\otimes u\otimes 1_{\D^{\otimes \infty}})(\sum_i a_i\otimes p\otimes q_i'\otimes 1_{\D^{\otimes \infty}})(1_A\otimes u^*\otimes 1_{\D^{\otimes \infty}})-\sum_i a_i\otimes q_i'\otimes p\otimes 1_{\D^{\otimes \infty}}\|<1/2\,.$$
Thus, since $p$ is a non-trivial projection of $\D$ we have that
 \begin{align*} (\text{Id}_A\otimes\Delta_p)^*[q]& =(\text{Id}_A\otimes\Delta_p)^*[q'\otimes 1_{\D^{\otimes \infty}}] = [(\text{Id}_A\otimes\Delta_p)(q'\otimes 1_{\D^{\otimes \infty}})] \\ 
 & = [\sum_i a_i\otimes p\otimes q_i'\otimes 1_{\D^{\otimes \infty}}] =[\sum_i a_i\otimes q_i'\otimes p\otimes 1_{\D^{\otimes \infty}}] \\ & < [\sum_i a_i\otimes q_i'\otimes 1_{\D^{\otimes \infty}}]=[q'\otimes 1_{\D^{\otimes \infty}}]=[q]\,.
 \end{align*}
 Therefore, $\text{Id}_A\otimes\Delta_p$ contracts projections of $A\otimes \D^{\otimes \infty}$. Moreover, given any ideal $I$ of $A$ we have that $\alpha_{A\otimes \D^{\otimes \infty}/I\otimes \D^{\otimes \infty}}$ contracts projections of $A\otimes \D^{\otimes \infty}/I\otimes \D^{\otimes \infty}$.
\end{exem}

\begin{theor}\label{theor_purely_inf} 
Let $A$ be a unital separable  $C^*$-algebra of real rank zero, let $\alpha\in \text{End }(A)$ be an injective endomorphism such that:
\begin{enumerate}
\item It satisfies the residual Rokhlin* property.
\item Given any ideal $I\in \mathcal{I}^{\alpha}(A)$  we have that $\alpha_{A/I}$ contracts projections of $A/I$.
\end{enumerate}
Then $A\times_{\alpha}\N$ is a purely infinite $C^*$-algebra.  
\end{theor}
\begin{proof}
It is enough to check that  $\bar{A}\times_{\bar{\alpha}} \Z$ is purely infinite, since it is strongly Morita equivalent to  $A\times_{\alpha}\N$. 

First observe that since $A$ is a separable $C^*$-algebra with  real rank zero so is $\bar{A}$, and since $\bar{\alpha}$ has the residual Rokhlin* property then the action of $\Z$ on $\widehat{(\bar{A})}$ is essentially free.

Now we claim that given any $I\in \mathcal{I}^\alpha(A)$ we have that $\bar{\alpha}_{\bar{A}/\bar{I}}$ contracts projections of $\bar{A}/\bar{I}$. Indeed, first observe that $\overline{\alpha_{A/I}}=\bar{\alpha}_{\bar{A}/\bar{I}}$. Then, it is enough to check that $\bar{\alpha}$ contracts projections of $\bar{A}$. Given any projection $p\in \bar{A}$ there exists a projection $q\in A$ and $n\in \N$ such that $\alpha_{1,n}(q)\sim p$, and hence $\bar{\alpha}(\alpha_{1,n}(q))=\alpha_{1,n}(\alpha(q))$. But since $\alpha$ contracts projections of $A$ we have that $\alpha^*([q])<[q]$, and hence 
$$\bar{\alpha}^*[p]=\bar{\alpha}^*[\alpha_{1,n}(q)]=[\alpha_{1,n}(\alpha(q))]<[\alpha_{1,n}(q)]=[p]\,,$$
as desired.

Therefore, by  \cite[Proposition 2.11]{PR} it is enough to check that every non-zero  hereditary sub-$C^*$-algebra in any quotient of $\bar{A}\times_{\bar{\alpha}}\Z$ contains an infinite projection. Since $\bar{A}$ separates ideals in  $\bar{A}\times_{\bar{\alpha}} \Z$, we have that every ideal of $\bar{A}\times_{\bar{\alpha}} \Z$ is of the form $\bar{I}\times_{\bar{\alpha}} \Z$, so $(\bar{A}\times_{\bar{\alpha}} \Z)/(\bar{I}\times_{\bar{\alpha}} \Z)\cong \overline{A/I}\times_{\overline{\alpha_{A/I}}} \Z$. Because of  all the assumptions pass to quotients, we can replace $\overline{A/I}$ by $\bar{A}$ and $\overline{\alpha_{A/I}}$ by $\bar{\alpha}$. 

Let $x$ be any positive element of $\bar{A}\times_{\bar{\alpha}} \Z$. By \cite[Lemma 3.2]{RS} there exists $a\in \bar{A}_+$ such that $a\lesssim x$. As $\bar{A}$ has real rank zero, there exists a projection $p\in \bar{A}$ such that $p\lesssim a$, so $p\lesssim x$. 
Then, there exists $n\in \N$ such that $\alpha^{*n}([p])<[p]\,$ in $V(\bar{A})$. Hence, there exists $t\in \bar{A}$ such that $t\bar{\alpha}^n(p)t^*+z=p$ for some idempotent $0\neq z\in \bar{A}$. If we set $r:=t\alpha^n(p)$ and we define $s:=rU^np$, where $U$ is the unitary in $M(\bar{A}\times_{\bar{\alpha}} \Z)$ that implements $\bar{\alpha}$, then
$$s^*s=(pU^{*n}r^*)(rU^np)=pU^{*n}r^*rU^np=pU^{*n}\bar{\alpha}^n(p)U^np=p$$
and
$$ss^*=(rU^np)(pU^{*n}r^*)=rU^npU^{*n}r=r\bar{\alpha}^n(p)r^*=t\bar{\alpha}^n(p)t^*.$$
Hence, $ss^*+z=t\bar{\alpha}^n(p)t^*+z=p$, whence $p$ is an infinite projection (since $z\neq 0$). 

Therefore, since $p\lesssim x$, by \cite[Proposition 2.6]{KR} there exist $\delta>0$ and $v\in \bar{A}\times_{\bar{\alpha}} \Z$ such that $v^*(x-\delta)_+v=p$. With $w=(x-\delta)_+^{1/2}v$ it follows that $w^*w=p$, whence $ww^*=(x-\delta)_+^{1/2}vv^*(x-\delta)_+^{1/2}\in \overline{x(\bar{A}\times_{\bar{\alpha}} \Z)x}$ is a projection equivalent to $p$. Hence, $ww^*$ is infinite because so is $p$. Thus, by \cite[Proposition 2.11]{PR}, we have that $\bar{A}\times_{\bar{\alpha}} \Z$ is a purely infinite $C^*$-algebra and hence $A\times_\alpha \N$ is too.
\end{proof}

\section{Examples}
In this section we will use the previous results to construct interesting examples of purely infinite crossed products $C^*$-algebras with different ideal structures. In particular we are going to build actions on strongly purely infinite $C^*$-algebras $A$, i.e. $A\cong A\otimes \mathcal{O}_\infty$. Since $\mathcal{O}_\infty$ is a strongly self-absorbing $C^*$-algebra with non-trivial projections, when $A$ has finitely many ideals we can apply  Corollary \ref{corol_rokhlin} to perturb any injective endomorphism $\beta\in \text{End }(A)$ to another endomorphism $\alpha$ satisfying the residual Rokhlin* property and with $K_*(\alpha)=K_*(\beta)$. Therefore $A\times_\alpha \N$ will be a purely infinite $C^*$-algebra with a lattice isomorphism  $\Phi:\mathcal{I}(A\times_\alpha\N)\longrightarrow \mathcal{I}^\beta(A)$. 

\begin{exem}\label{exem_2} Let $\mathcal{O}_n$ be the Cuntz algebra for $n=2,\ldots,\infty$;  by Kirchberg's work we have that $\mathcal{O}_n\otimes\mathcal{O}_\infty\cong \mathcal{O}_n$. Given $m\in\N$ with $m<n$, by Corollary  \ref{corol_rokhlin} there exists an injective endomorphism $\alpha_m:\mathcal{O}_n\longrightarrow   \mathcal{O}_n$ that satisfies the  Rokhlin* property and $K_*(\alpha_m)=m\cdot \text{Id}_{K_*(\mathcal{O}_n)}$. Therefore, since $\mathcal{O}_n$ is a simple purely infinite $C^*$-algebra in the UCT class, so is $\mathcal{O}_n\times_{\alpha_m}\N$. Thus, by Kirchberg-Philllips classification results, it is enough to compute its $K$-theory to determine in which isomorphism class lies. 

First recall that 
$$K_*(\mathcal{O}_n)\cong \left\lbrace \begin{array}{ll}  (\Z,0) & \text{if }n=\infty \\ (\Z/(n-1)\Z,0) & \text{otherwise}\end{array} \right. \,.$$ 
Observe that the endomorphism $\alpha_m$ defines a group endomorphism $K_0(\alpha_m):K_0(\mathcal{O}_n)\longrightarrow K_0(\mathcal{O}_n)$ given by $y\longmapsto my$ for every $y\in K_0(\mathcal{O}_n)$.  By continuity of $K$-theory we have that 
$$K_*(\overline{\mathcal{O}_n})\cong\varinjlim_{K_0(\alpha_m)}K_i(\mathcal{O}_n)\,,$$
so
$$\varinjlim_{K_0(\alpha_m)}K_0(\mathcal{O}_n)=\left\lbrace \begin{array}{ll}  \Z[1/m] & \text{if }n=\infty \\ \Z/k\Z & \text{if }n\neq\infty \text{ where }k= [gcd(n-1,m)|(n-1)]\end{array} \right. \,,$$
where given $a,b\in \N$ we define $[a|b]:=\text{max }\{c\in\N: c|b \text{ and }gcd(a,c)=1\}$. Then, $K_0(\overline{\alpha_m}):K_0(\overline{\mathcal{O}_n})\longrightarrow K_0(\overline{\mathcal{O}_n})$ is given by   $x\longmapsto mx$ for every $x\in K_0(\overline{\mathcal{O}_m})$, while $K_1(\overline{\mathcal{O}_m})=0$. 

Therefore, we can assume that $m\leq k$ and hence 
$$\text{Ker }(\text{Id}-K_0(\overline{\alpha_m}))=\left\lbrace \begin{array}{ll}  \Z & \text{if }n=\infty\text{ and }m=1 \\ 0 & \text{if }n=\infty\text{ and } m\neq 1\\ \Z/k\Z & \text{ if }n\neq \infty \text{ and }m=1 \\  \Z/l\Z & \text{ otherwise, where }l= k/gcd(k,m-1)\end{array} \right. \,,$$
and 
$$\text{Coker }(\text{Id}-K_0(\overline{\alpha_m}))=\left\lbrace \begin{array}{ll}  \Z & \text{if }n=\infty\text{ and }m=1 \\ \Z/(m-1)\Z & \text{if }n=\infty \text{ and }m\neq 1\\ \Z/k\Z & \text{if }n\neq \infty \text{ and }m=1 \\ \Z/l\Z & \text{otherwise,  where }l=k/gcd(k,m-1)\end{array} \right. \,.$$
So, using the Pimsner-Voiculescu six-term exact sequence,  we have 
$$\xymatrix{{K_0(\overline{\mathcal{O}_n})}\ar[r]^{\text{Id}-K_0(\overline{\alpha_m})} & {K_0(\overline{\mathcal{O}_n})}\ar[r]& K_0(\mathcal{O}_n\times_{\alpha_m}\N)\ar[d] \\ K_1(\mathcal{O}_n\times_{\alpha_m}\N) \ar[u] & {0}\ar[l]& {0}\ar[l]_{0 } }\,.$$
Thus, $K_0(\mathcal{O}_n\times_{\alpha_m}\N)\cong \text{Coker }(\text{Id}-K_0(\overline{\alpha_m}))$  and  $K_1(\mathcal{O}_n\times_{\alpha_m}\N)\cong \text{Ker }(\text{Id}-K_0(\overline{\alpha_m}))$. Hence, by Kirchberg-Phillips Classification Theorems  it follows that
$$\mathcal{O}_n\times_{\alpha_m}\N\cong \left\lbrace \begin{array}{ll}   B & \text{if }n=\infty \text{ and }m=1  \\ \mathcal{O}_{m} & \text{if }n=\infty \text{ and }m\neq 1 \\ \mathcal{O}_{l+1}\otimes \mathcal{O}_{l+1} & \text{otherwise}\end{array} \right. \,,$$
where $B$ is the unique Kirchberg algebra with $K_*(B)\cong (\Z,\Z)$.
\end{exem}

\begin{exem}
Let $E$ be the following graph
$$
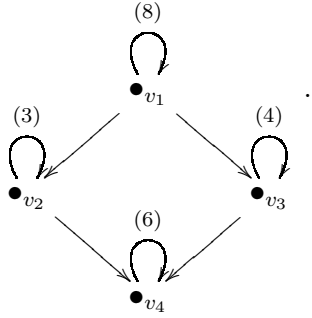
\xymatrix{  & \bullet_{v_1}\uloopr{}^{(8)}\ar[dl]\ar[dr]& \\  \bullet_{v_2}\uloopr{}^{(3)}\ar[dr]& & \bullet_{v_3}\uloopr{}^{(4)}\ar[dl]  \\& \bullet_{v_4}\uloopr{}^{(6)}& }\,.$$
Then, the graph $C^*$-algebra $C^*(E)$ is purely infinite with real rank zero \cite{HS}, and  $C^*(E)\cong C^*(E)\otimes \mathcal{O}_\infty$. So, taking $\beta=\text{Id}_{C^*(E)}$, by Theorem \ref{rokhlin_ideals} there exists an injective endomorphism $\alpha:C^*(E)\longrightarrow C^*(E)$ that  satisfies the residual Rokhlin* property, and $\mathcal{I}^\alpha(C^*(E))=\mathcal{I}(C^*(E))$. Moreover, $K_*(\alpha)=\text{Id}_{K_*(C^*(E))}$.

$E$ has the following hereditary and saturated subsets 
 $$\{\emptyset,\{v_4\},\{v_4,v_2\},\{v_4,v_3\},\{v_2,v_3,v_4\},\{v_1,v_2,v_3,v_4\}\}\,.$$
so by \cite{Bates} the Hasse diagram of its primitive ideal space $X_3$ is
$$
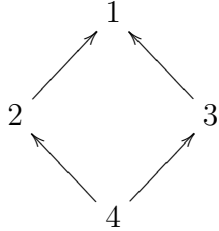
\xymatrix{  & {1} & \\  {2}\ar[ur]& & {3}\ar[ul]  \\& {4}\ar[ur]\ar[ul] & }$$  
Given two ideals  $I,J\in \mathcal{I}(C^*(E))$ with $I\subseteq J$, using \cite{Bates2} we can easily compute $K_*(J/I)$. For example:
$$K_*(I_{\{4\}})\cong(\Z/5\Z,0)\,,\qquad K_*(I_{\{3,4\}})\cong(\Z/5\Z\oplus \Z/3\Z ,0)\,,$$
$$K_*(I_{\{2,4\}})\cong(\Z/5\Z\oplus \Z/2\Z,0)\,,\qquad K_*(I_{\{2,3,4\}})\cong(\Z/5\Z\oplus \Z/3\Z\oplus \Z/2\Z ,0)\,,$$
$$K_*(I_{\{1,2,3,4\}})\cong(\Z/7\Z\oplus\Z/5\Z\oplus \Z/3\Z\oplus \Z/2\Z ,0)\,,$$ 
where $I_X:=I_{Z}/I_Y$ and $Z,Y\subseteq E^0$ are hereditary and saturated subsets with $Y\subseteq Z$ and $X=Z\setminus Y$. 

Then $C^*(E)\times_\alpha \N$ is purely infinite  with primitive ideal space $X_3$. Since $C^*(E)$ separates the ideals of $C^*(E)\times_\alpha\N$, we have that all the subquotients of $C^*(E)\times_\alpha\N$  are of the form $\langle J\rangle / \langle I\rangle$ for $I,J\in \mathcal{I}(C^*(E))$ with $I\subseteq J$, and $\langle J\rangle / \langle I\rangle$ is strongly Morita equivalent to $(\bar{J}/\bar{I})\times_{\bar{\alpha}}\Z$. Therefore we can use the Pimsner-Voiculescu  six-term exact sequence for $K$-Theory to deduce that $K_*(\langle J\rangle / \langle I\rangle)\cong (K_0(J/I),K_0(J/I))$.

Finally observe that, given any ideals $I, J\in \mathcal{I}(C^*(E))$ with $I\subsetneq J$, does not exist any non-zero map $K_0(\langle J\rangle / \langle I\rangle)\longrightarrow K_1(\langle J\rangle / \langle I\rangle)$, thus $C^*(E)\times_\alpha \N$ is $K_0$-liftable. Hence, by results of Pasnicu and R{\o}rdam \cite[Theorem 4.2]{PR}, it must have real rank zero.
\end{exem}

\section*{Acknowledgments}

Parts of this work was done during visits of the first author to the Departamento de Matem\'aticas de la Universidad de C\'adiz (Spain) and of the second author to the Institutt for Matematiske Fag, Norges Teknisk-Naturvitenskapelige Universitet (Trondheim, Norway). The authors thank both host centers for their warm hospitality.

\end{document}